\documentclass[11pt]{amsart}

\usepackage{graphicx}
\usepackage{amssymb}
\usepackage{epstopdf}
\usepackage{amsmath}
\usepackage{amscd}
\usepackage{amssymb}
\usepackage{amsfonts}

\usepackage{latexsym}
\usepackage{verbatim}
\usepackage{epsfig}
\usepackage[all]{xy}

\usepackage{color}
\definecolor{DarkRed}{rgb}{0.8,0.04,0.02}
\definecolor{DarkGrey}{rgb}{0.4,0.4,0.4}

\numberwithin{equation}{section}
\newtheorem{lemma}[equation]{Lemma}
\newtheorem{proposition}[equation]{Proposition}
\newtheorem{theorem}[equation]{Theorem}
\newtheorem{corollary}[equation]{Corollary}

\theoremstyle{remark}

\newtheorem{remark}[equation]{Remark}
\newtheorem{definition}[equation]{Definition}

\newcommand\eps{\epsilon}

\newcommand\ZZ{\mathbb Z}
\newcommand\NN{\mathbb N}

\newcommand\Prob{{\mathbb P}}

\newcommand\pe{\mathbb P_{\textnormal{erg}}}

\newcommand\pen{\mathbb P_{\textnormal{erg}}'} 

\begin{document}

\title{Borel Isomorphism of SPR Markov Shifts}

\author{Mike Boyle}
\address{Department of Mathematics - University of Maryland}
\email{mmb@math.umd.edu}
\author{J\'er\^ome Buzzi}
\address{Laboratoire de Math\'ematiques d'Orsay (CNRS, UMR 8628) - Universit\'e Paris-Sud}
\email{jerome.buzzi@math.u-psud.fr}

\author{Ricardo G\'omez}
\address{Instituto de Matem\'aticas - Universidad Nacional Aut\'onoma de M\'exico
}
\email{rgomez@math.unam.mx}




\begin{abstract}
We show that strongly positively recurrent Markov shifts (including shifts of finite type) are classified up to Borel conjugacy by their entropy, period and their numbers of periodic points.
\end{abstract}

\maketitle



\section{Introduction}

Theorem \ref{Hochiso} below is one of the results in 
the \lq\lq full sets\rq\rq\    paper of Hochman
\cite{Hochman}. In the statement, 
\lq Markov shift\rq\    means countable state 
Markov shift. The free part of a Borel system is the 
subsystem obtained  by restriction to the nonperiodic points, 
and a full subset is an invariant subset 
of measure one for every
invariant Borel probability measure. 
Two Borel systems are {\it almost-Borel isomorphic} if they are 
Borel isomorphic after restriction to full subsets of their free parts. 
Detailed definitions for the Introduction are given in the next section.

\begin{theorem} \cite{Hochman} \label{Hochiso} 
Two mixing Markov shifts are almost-Borel  isomorphic if and
only if (1) they have equal entropy and (2) one has a measure 
of maximum entropy if and only the other does.  
\end{theorem} 

An important observation \cite{Hochman} in this 
setting is that 
two Borel systems that embed 
each into the other are Borel isomorphic, by 
a Borel variant of
Cantor-Bernstein Theorem (a.k.a.\ the 
measurable Schr\"oder-Bernstein Theorem). 
 Consequently 
Theorem \ref{Hochiso} was an immediate corollary of 
 the following embedding theorem.

\begin{theorem}\cite{Hochman}\label{Hochmanembed}   Suppose $(Y,T)$ is a mixing 
Markov shift and $(X,S)$ is a Borel system such that 
$h(S,\mu )< h(T)$ for every ergodic invariant Borel probability
$\mu$
on $X$. Then there is an 
almost-Borel 
embedding of $(X,S)$ into 
$(Y,T)$. 
\end{theorem}  

This theorem easily leads to a decisive almost-Borel 
classification of Markov shifts, and has implications for 
other systems  \cite{Hochman, BB2014}.

The study of 
Borel dynamics, 
adopting weakly  wandering sets 
as the relevant notion of negligible sets, 
was initated  by   Shelah and Weiss \cite{ShelahWeiss,Weiss1, Weiss2}.   
Here that notion of isomorphism preserves additionally the 
infinite and quasi-invariant 
measures (and again it is natural to restrict to free parts).  
Whether there is a theorem for Borel dynamics 
like Theorem \ref{Hochmanembed} 
is a difficult open problem, discussed in \cite{Hochman}.  
Our purpose in this paper is to show that a generalization 
 of Theorem \ref{Hochiso} to this  richer category holds 
in at least one meaningful case.

\begin{theorem}\label{spriso} 
The free parts of mixing SPR Markov shifts are Borel isomorphic  if and only if they have equal entropy.  
\end{theorem} 

We note that Hochman \cite{Hochman} has asked 
if  those free parts are in fact topologically conjugate, at least in the case of subshifts of finite type.

As in the almost-Borel case, Theorem \ref{spriso} 
is an immediate corollary of 
an embedding result, stated next.  

\begin{theorem}\label{MSEmbed2} 
Suppose $(Y,T)$ is a mixing SPR Markov shift 
and $(X,S)$ is a Markov shift such that 
$h(X)=h(Y)$ and 
$X$ has  a unique irreducible component of 
full entropy and this component is a mixing SPR Markov shift. 
Then there is a Borel embedding of $(X,S)$ into 
$(Y,T)$. 
\end{theorem} 

The proof is independent of Hochman's result and techniques. 
Roughly speaking, Hochman builds almost-Borel embeddings from the bottom 
up with a uniform version of the Krieger Generator Theorem
\cite{KriegerGenerator}.  In our much more special situation, 
we can build Borel embeddings with the following offshoot 
of the Krieger Embedding Theorem. 

\begin{theorem}\label{MSEmbed1} 
Suppose $(Y,T)$ is a mixing Markov shift 
and $(X,S)$ is a Markov shift such that 
$h(X)<h(Y)$.  
Then there is a Borel embedding of the free part of $(X,S)$ into 
$(Y,T)$. 
\end{theorem}

Theorem \ref{MSEmbed1}, though not completely trivial, 
is completely unsurprising. (The question of when 
a Markov shift embeds {\it continuously} into a mixing 
Markov shift is much harder 
\cite{Fiebigs1997,Fiebigs2005}.)  
The  novel 
feature in the proof of Theorem \ref{MSEmbed2} 
is the 
use  of a ``top-down'' embedding given by the almost 
isomorphism theorem of \cite{BBG2006} to reduce the problem 
to embeddings of lower entropy systems. 

At the end of the paper we state the Borel classification of the free parts of irreducible SPR Markov shifts,  
which follows from the mixing case. 

\subsection*{Acknowledgments}
We thank Mike Hochman for the stimulating discussions 
out of which this paper emerged. 
M. Boyle gratefully acknowledges the
financial support of ANR project DynNonHyp BLAN08-2\_313375 and the hospitality of the Mathematics Department of Universit\'e Paris-Sud in Orsay.

\section{Definitions and background}\label{sec:def}


A {\bf Borel system} $(X,\mathcal X,T)$ is a standard Borel
space\footnote{$\mathcal X$ is a $\sigma$-algebra of subsets of $X$
  such that there is distance on $X$ which turns it into a complete
  separable space whose collection of Borel subsets is $\mathcal X$.}
$(X,\mathcal X)$ together with a Borel automorphism\footnote{A
  bijection such that 
$T^{-1}\mathcal X:=\{T^{-1}E:E\in\mathcal  X\}=T\mathcal X=\mathcal X$.} $T:X\to X$. We often abbreviate
$(X,\mathcal X,T)$ to $(X,T)$ or $X$ or $T$ if it does not create
confusion. 
A {\bf Borel factor map} is a homomorphism of Borel systems:
a (not necessarily onto) Borel
measurable 
map intertwining the actions. An 
 isomorphism or conjugacy of Borel
systems is a bijective Borel factor map; an embedding of Borel systems 
is an injective Borel factor map. 
%
By an easy exercise in descriptive set theory
(see \cite[p.399]{Weiss1}),
there is a Borel conjugacy of two systems if and only if there 
is a Borel conjugacy between their free parts and for each $n$ the
cardinalities 
of their sets of periodic orbits of size $n$ is the same.


Given a Borel system $(X,T)$, we use 
$\Prob(X)\supset \pe(X) \supset \pen(X)$ 
respectively to denote 
the sets of all  measures\footnote{Unless specified otherwise, the
  word measure will denote an invariant Borel probability.},
all ergodic measures,  and  all ergodic nonatomic 
measures. 
%
Recall from \cite{Weiss1} that a set $W$ is {\bf wandering} if it is
Borel and if $\bigcup_{k\in\ZZ} T^kW$ is a disjoint union (which we
denote $\bigsqcup_{k\in\ZZ} T^kW$).  A set is {\bf weakly wandering}
if it is a Borel subset of a countable union of wandering sets. Such
a set has measure zero for all quasi-invariant 
measures \cite{ShelahWeiss,Weiss1}, not only for measures in $ \Prob(X)$. 
To avoid any mystery, we record 
a simple remark. 
\begin{remark} \label{lem:rug} 
Suppose $(X,S)$ and $(Y,T)$ are Borel systems and 
each contains an uncountable Borel set which is wandering. Then 
the systems are Borel isomorphic if and only if they are Borel 
isomorphic modulo wandering sets. 
\end{remark} 
The basis of the remark is the following. 
Any weakly wandering set is contained in the orbit of a wandering 
set. Under the assumption, such wandering sets in $X$ and $Y$ 
  can be enlarged to 
uncountable Borel subsets of the ambient Polish space. Any two 
such sets are Borel isomorphic.


%
%

A {\bf Markov shift} $(X,S)$ is a topological system $\Sigma (G)$  defined by the action of the
left shift $\sigma:(x_n)_{n\in\ZZ} \mapsto (x_{n+1})_{n\in\ZZ}$ on the
set $\Sigma(G)$ of paths on some oriented
graph $G$ with countably (possibly finitely)
many vertices and edges. We will use the edge shift  
(rather than the vertex shift) presentation. 
The domain $X$ is 
the set of  $x=(x_n)_{n\in \ZZ}  \in \mathcal E^\ZZ$ (where $\mathcal E$ is the set of
oriented edges) such that for all $n$, the terminal vertex of $x_n$
equals the initial vertex of $x_{n+1}$. 
The (Polish) topology on $X$ is the relative topology of the product 
of the discrete topologies. 
When
$G$ is finite, $\Sigma (G)$ is a shift of finite type (SFT).  
$\Sigma (G)$ is
{\bf irreducible} if $G$ contains a unique strongly connected
component, i.e., a maximal set of the vertices such that for any pair,
there is a loop containing both. 
An arbitrary Markov shift is the disjoint 
union of a wandering set and countably many disjoint irreducible
Markov shifts. 
An irreducible Markov shift is mixing if and only if the g.c.d. of the
periods of 
its periodic points is 1. 

The Borel entropy of a system $(X,S)$ is the supremum of the
Kolmogorov-Sinai entropies $h(S, \mu )$, $\mu \in \Prob(X)$. 
Markov shifts of positive entropy contain  uncountable wandering 
sets; so, by the Remark \ref{lem:rug}, for simplicity we
can neglect weakly wandering 
sets in both statements and proofs. 
An irreducible Markov shift $(X,S)$ (more generally, an irreducible component) has at most one measure of maximum
(necessarily finite)  entropy
\cite{Gurevic1970};  
if this measure $\mu$ exists,
then 
$(S, \mu )$ is measure-preservingly isomorphic to the product of a  finite entropy Bernoulli
shift and a finite cyclic rotation (see \cite{BB2014} 
for comment and references).
 
An irreducible Markov shift 
$\Sigma$ is 
 {\bf strongly positively recurrent} (or 
{\bf stably positive recurrent} or just {\bf SPR}) 
if it admits a measure $\mu$ of maximal  entropy 
which is {\it exponentially recurrent}: 
for every non-empty open subset $U\subset\Sigma$, 
 $$
   \limsup_{n\to\infty} \frac1n\log\mu 
\Big(  \Sigma\setminus\bigcup_{k=0}^{n-1}\sigma^{-k}U \Big) < 0\ .
 $$
We refer to \cite{BBG2006, Gurevich1996, GurevichSavchenko} for more on SPR shifts. 
In the language of \cite{Gurevich1996, GurevichSavchenko}, 
the SPR Markov shifts are the positively recurrent symbolic 
Markov chains defined by stably recurrent matrices (further developed 
in \cite{GurevichSavchenko} as the fundamental class of \lq\lq stably
positive\rq\rq\   matrices). 
The SPR Markov shifts are a natural subclass preserving some of the 
significant properties of finite state shifts \cite[Sec.2]{BBG2006}.

\section{Embedding a Markov shift with smaller entropy}
In this section we 
will prove Theorem \ref{MSEmbed1}.
First we recall and adapt some standard finite-state symbolic dynamics 
(for more detail on this, see \cite{Boyle1983} or
\cite{LindMarcus1995}). 

\begin{lemma}\label{lem:disjSFT}
Suppose $\epsilon
> 0$ and 
$X$ is  a mixing Markov shift with entropy $h(X)>0$. 
 Then $X$  
contains infinitely many mixing SFTs $S_n$, 
pairwise disjoint, such that 
 $h(S_n)>h(X)-\epsilon$ for all $n$. 
\end{lemma}
\begin{proof} 
$X$ contains an SFT $S$ with entropy greater than $h(X)-\epsilon$ \cite{Gurevic1970}; 
$S$ is easily enlarged to a mixing SFT $S'$ in $X$. 
The complement of a given  proper subshift of $S'$ contains a mixing SFT 
with entropy arbitrarily close to $h(S')$ \cite[Lemma 26.17]{DGS}. 
Thus one 
can construct the required family inductively. 
\end{proof} 

\begin{definition}
For a system $(X,S)$, $|P^o_n(X)|$ denotes the cardinality of the set of
points in $S$-orbits of length $n$. 
\end{definition} 

\begin{theorem}[Krieger Embedding Theorem \cite{KriegerEmbedding}]
\label{KriegerEmbed}
Let $X$ be a subshift on a finite alphabet and  $Y$  a mixing
SFT 
such that $h(X)<h(Y)$ and 
$|P^o_n(X)| \leq |P^o_n(Y)|$  
for all $n$. 
Then there is a continuous embedding of $X$ into $Y$.
\end{theorem}

\begin{proposition}\cite[Lemma 2.1 and p.546]{Boyle1983} \label{Boyle1983}
Suppose  $X$ is a mixing SFT and $M$ is a positive integer. 
 Let $\mathcal O_1, \dots , \mathcal O_r$ be distinct finite orbits in
 $X$. 
 Let $W_i$ be the set of points whose positive iterates 
are positively asymptotic to
$\mathcal O_i$, and let $W=\cup_i W_i$.
Then there exist a mixing SFT $Z$ and a continuous surjection $p:Z\to X$ such that:
 \begin{enumerate}
  \item $|p^{-1}(x)|=1$ for all $x$ outside $W$
  \item The preimage of $\mathcal O_i$ is an orbit $\widetilde{\mathcal O_i}$ of
    length $M|\mathcal O_i|$. 
  \item $p^{-1}(W_i)$ is the 
set of points positively asymptotic to $\widetilde{\mathcal O_i}$. 
\end{enumerate}
\end{proposition}

\begin{corollary}\label{coro:BoyleEmbed}
Let $X$ and $Y$ be SFTs such that $h(X)<h(Y)$ and $Y$ is mixing. Then
there is a continuous embedding of $X\setminus X_0$ into $Y$ where
$X_0$ is the union of a weakly wandering set and a finite set of periodic points.
\end{corollary}
\begin{proof} 
We have that $\lim_n (\, |P^o_n(Y)|-|P^o_n(X)|\, ) = \infty$. 
Thus we may choose $M$ to
 build $Z$ as in  Proposition \ref{Boyle1983} such that $Z$, 
by Theorem \ref{KriegerEmbed}, embeds into $Y$. 
The map $Z\to X$ is a Borel isomorphism on the complement of 
a set $X_0$ of points positively  asymptotic to finitely  many periodic points. 
\end{proof} 

To reduce  Theorem \ref{MSEmbed1} to this corollary, we use 
reductions stated as three lemmas. 
A {\bf loop system} is a Markov shift defined by a
{\bf loop graph}: a graph made of simple loops which are 
based at a common
vertex and otherwise do not intersect. 
Given a power series $f = \sum_{n=1}^{\infty} f_nz^n$ 
with coefficients in $\ZZ_+$, 
we let $\Sigma_f$ denote the  loop system 
with exactly $f_n$ simple
loops 
of length $n$ in the loop graph.  
If $h(\Sigma_f) = \log \lambda < \infty$, then 
\begin{enumerate} 
\item 
$0< f(1/\lambda ) \leq 1$,  
\item $\alpha < \lambda \implies f(1/\alpha ) = \infty$ and 
\item $f(1/\lambda )=1 $ 
if $\Sigma_f$ has a measure of maximum 
entropy (i.e. is positive recurrent). 
\end{enumerate}
For more on loop systems and Markov shifts, see 
\cite{BBG2006,GurevichSavchenko,Kitchens1998} and 
their references.

\begin{lemma}\label{lem:RedLoop}
Any Markov shift $X$ is Borel isomorphic to a Borel system 
 $$
W \ \sqcup\ \bigsqcup_{n\in\NN} \Sigma(L_n) 
 $$
where $W$ is weakly wandering and for each $n$, 
$L_n$ is a loop
graph. 
\end{lemma}


\begin{lemma}\label{lem:RedSFT}
Let $\Sigma$ be a loop system and $h>h(\Sigma)$. Then there is a SFT $S$ with $h(S)<h$ such that $\Sigma$ has a continuous embedding into $S$.
\end{lemma}


%
Before proving the lemmas, we deduce 
the lower-entropy embedding theorem from them.

\begin{proof}[Proof of Theorem \ref{MSEmbed1}]
According to 
Remark \ref{lem:rug} and Lemma \ref{lem:RedLoop},
we may assume that $X$ is a disjoint union of loop systems 
$\Sigma (L_n)$. 
Let $h=(h(Y) +h(X))/2>h(X)$. 
By Lemma \ref{lem:RedSFT}, each loop system $\Sigma (L_n)$ can be 
(continuously) embedded into some SFT $W_n$ with entropy less than
$h$.
Let $\epsilon =h(Y)-h>0$. 
By Lemma \ref{lem:disjSFT}
(with $\epsilon = (h(Y)-h))/2)$),
 there are pairwise disjoint mixing SFTs $Y_n$ in
$Y'$ with $h(Y_n)> h$. 
Finally, Corollary \ref{coro:BoyleEmbed} shows that each
 $W_n$ (apart from finitely many periodic points) 
can be Borel embedded into $Y_n\subset Y$. Altogether, 
apart from a countable set of periodic points, $X$ has 
been Borel embedded into $Y$. 
\end{proof}

We now prove the lemmas.

\begin{proof}[Proof of Lemma \ref{lem:RedLoop}]
Let $G$ be some graph presenting $X$. For convenience, we identify its vertices with $1,2,\dots$. Observe that each $W^{\eps}_n:=\{x\in X:x_0=n$ and $\forall i>0\; x_{\eps i}\ne n\}$ ($n\in\NN^*,\eps\in\{-1,+1\}$) is wandering. Consider the loop graphs $L_n$ defined by the first return loops of $G$ at vertex $n$ which avoid the vertices $k<n$.

For each $x\in X$, let $N:=\inf\{n\geq1:\exists a_k,b_k\to\infty\;
x_{-a_k}=x_{b_k}=n\}$ and consider the following three cases. 
\begin{enumerate}
\item 
$N=\infty$. Then  there exists $\eps\in\{-1,+1\}$ such that 
$x\in\sigma^{-j}W_{x_0}^\eps$, where 
$j:=\eps\sup\{\eps i\in\ZZ:x_i=x_0\}\in\ZZ$. 
\item 
 $N<\infty$ and $\{x_m : m \in \ZZ \}\cap [1,N) \neq \varnothing$.
   Then  there exist $k\in [1,N)$ 
and $\eps\in\{-1,+1\}$ such that
$j:=\eps\sup\{\eps i\in\ZZ:x_i=k\}\in\ZZ$, so
 $x\in\sigma^{-j}W_{k}^\eps$. 
\item
Otherwise, $x\in\Sigma(L_N)$.
\end{enumerate}
To conclude, observe that $\bigcup_{k\in\NN^*,j\in\ZZ,\eps\in\{-1,+1\}} \sigma^{-j}W_k^{\eps}$ is a weakly wandering set.
\end{proof}

\begin{proof}[Proof of Lemma \ref{lem:RedSFT}] 

Let $\Sigma = \Sigma_f$, a loop system described by a power series  
$f = \sum_{n=1}^{\infty} f_nz^n$. 
 If $f$ is a polynomial, then $\Sigma_f$ is itself
an SFT. 
From now on, we assume $f$ to have infinitely many non-zero terms.

We are going to build the SFT as a finite loop system $\Sigma_p$, with a polynomial
$p$ obtained by truncating the power series $f$ and then adding some monomials to ensure 
enough space for the embedding while keeping the entropy $<h$.

Let $\beta\in(h(\Sigma),h)$. Given a positive integer 
$N$, let $f^{(N)}$ denote the truncation of $f$ to the 
polynomial $f_1z + f_2z^2 + \cdots + f_Nz^N$. 
As $f(e^{-h(\Sigma)})\leq1$ and $h(\Sigma) < \beta$ we have 
$f_n <e^{n\beta}$ for all $n\geq1$.
Let $g^{<N>}$ denote the polynomial 
$g_{N+1}z^{N+1}+ g_{N+2}z^{N+2}+ \cdots + g_{2N}z^{2N}$, 
where $g_n=\lceil e^{n\beta}\rceil$ (the integer ceiling). Then 
\begin{align*} 
|g^{<N>}(z)| 
&\leq \Big[(e^{(N+1)\beta} +1) +\cdots + (e^{2N\beta} +1)|z|^{N-1}\Big]|z|^{N+1} \\
 &= e^{(N+1)\beta}|z|^{N+1}\Bigg[\frac{1-(e^\beta |z|)^N}{1-e^\beta |z|}\Bigg] 
+ |z|^{N+1} \Bigg[\frac{1-| z|^N}{1- |z|}\Bigg] \ . 
\end{align*}
As $\beta>0$, we see that $\lim_{N\to\infty} g^{<N>}(z)=0$ uniformly for $|z|$ fixed, smaller
than $e^{-\beta}$.

Recall that $f(r)<1$ for $r<e^{-h(\Sigma)}$. Also if $r>0$ and $|z|=r$ and 
$f^{(N)}(r)<f(r)< 1$, then $|1-f^{(N)}(z)| \geq 1-f^{(N)}(r)> 1-f(r)>0$. 
Fix some $\gamma\in(\beta,h)$ and 
then $N$ sufficiently large that the following hold: 
\begin{enumerate} 
\item 
$|2g^{<N>}(z)| <1 -f(e^{-\gamma}) <  1-f^{(N)}(e^{-\gamma}) \leq 
|1-f^{(N)}(z)|$;
\item 
both $1-f^{(N)}(z)$ and $1-f^{(N)}(z)-2g^{<N>}(z)$ are non-zero. 
\end{enumerate} 
It follows from Rouch\'e's Theorem that $1-f^{(N)}$ and 
$1-f^ {(N)}-2g^{<N>}$ have the same number of zeros inside the 
circle $|z|=e^{-\gamma}$, i.e. no zeros. Thus, setting $p:=f^ {(N)}+2g^{<N>}$,
we get $h(\sigma_p) < \gamma < h$.

Now, set $k=g^{<N>} $ and
split $p$ as $p=(f^{(N)}+ g^{<N>}) + g^{<N>} =: h+k$  and 
let $q := h(1+ k+ k^2 + \cdots )$.  $\sigma_q$ is the loop system defined from $\sigma_{h+k}$
by replacing the loops from $k$ by all the loops made by concatenating a copy of a loop from $h$
with an arbitrary positive number of copies of loops from $k$
 (see \cite[Lemma 5.1]{BBG2006} for detail).
It follows that $\sigma_q$ can
be identified to the subset of $\sigma_p$ obtained by removing a copy of $\sigma_k$ with the points
asymptotic to it. 
Hence, there is a continuous embedding 
of $\sigma_q$ into $\sigma_p$. 

Note that for $n\leq N$ we have $f_n=p_n=q_n$. 
Also, for $n> N$, $f_n <e^{n\beta}\leq (1+k+k^2+\cdots )_n \leq
q_n$. This yields an embedding $\sigma_f \to \sigma_q$ 
and concludes the proof.
\end{proof}


\section{The SPR case}

We now give  the proof of Theorem \ref{MSEmbed2}. 
 Let $X'$ be the mixing SPR component
of $X$ with $h(X)= h(Y)$. Equal entropy mixing SPR Markov shifts 
are {\it almost isomorphic} as defined 
and proved in \cite{BBG2006}. 
Consequently there will be a word $w$ and a subsystem 
$\Sigma^w$ of $X'$ (consisting of the points which see 
$w$ infinitely often in the past and in the future) such that 
there is a continuous embedding 
$\psi_0$ from $X_0=\Sigma^w$ onto a subsystem 
$Y_0$ of $Y$ and $\epsilon >0$ such that the complements 
$ X'\setminus X_0$ and $Y\setminus Y_0$   
have Borel entropy less than $h(Y)-\epsilon$.  


The Borel subsystem  $ X\setminus X_0$ is (after passing to a higher block
presentation) the union of 
a Markov shift $X_1$ 
 (the subsystem of $X$ avoiding the word $w$) and a weakly
wandering set $W$ (defined by the occurence of $w$, with a 
failure of infinite recurrence in the past or future). 
By
Remark \ref{lem:rug}, we can
forget about $W$. 
We cannot expect $X_1$ to have entropy less than 
$h(Y\setminus Y_0)$, 
  and therefore we cannot apply 
Theorem \ref{MSEmbed1} to embed $ X_1$ 
into a subsystem of $Y\setminus Y_0$. Instead,  
we will push  $X_1$ into the image of $X_0$,
 and adjust the definition on $X_0$ to 
keep injectivity.

For $L$ large enough, 
\[\Sigma^{w,L}:=\{x\in\Sigma:\forall n\in\ZZ\ \exists
k\in\{0,\dots,L\}\; x_{n+k}\dots x_{n+k+|w|-1}=w\}\]
 is a mixing Markov subshift with 
$h(\Sigma^{w,L}) >h(X_1)$. We apply 
Lemma \ref{lem:disjSFT}
 to 
get pairwise disjoint mixing SFTs $Y_1,Y_2, \dots $ 
in $\Sigma^{w,L}$ satisfying $h(Y_i)>h(X_1)$ for all $i\in\NN$.

Let $C$ denote the complement in $X_1$ of the periodic points. 
Theorem \ref{MSEmbed1} gives Borel embeddings 
$\gamma_i: C\to Y_{i}$. 
Let $Z_i:=\gamma_i(C)\subset Y_i$ 
 and 
let $\phi_i$ be the conjugacy $\gamma_{i+1}\circ
  \gamma_i^{-1}:Z_{i}\to Z_{i+1}$.
We define $\psi: X_0\cup C \to\Sigma'$ by 
\begin{align*} 
\psi : x \ &\mapsto\ \gamma_1(x) \in Z_1 \quad \quad \quad \quad \textnormal{if } x\in C \\
&\mapsto \ \phi_i (\psi_0(x) ) \in Z_{i+1}\quad \ \textnormal{if } \psi_0( x)\in
Z_i\\ 
&\mapsto \ \psi_0 (x) \quad \quad \quad \quad \quad \quad \, \text{otherwise} \ .
\end{align*}
%
This  $\psi$ 
is a Borel embedding. 
This finishes the proof of Theorem \ref{MSEmbed2}. \qed 

Lastly we record the obvious corollary of Theorem \ref{spriso}.  
\begin{theorem}
The free parts of two irreducible SPR Markov shifts are Borel isomorphic 
if and only 
if they have the same  entropy and period.  
\end{theorem}

\bibliographystyle{plain} 
\bibliography{b-HAL}

\end{document}